\documentclass[reqno,12pt]{amsart}

\usepackage{amsmath}
\usepackage{amsfonts}
\usepackage{amssymb}
\usepackage{eufrak}
\usepackage{amscd}
\usepackage{amsthm}
\usepackage{epsfig}
\usepackage{amstext}
\usepackage[all]{xy}

\theoremstyle{plain}

 \newtheorem{theorem}{Theorem}

 \newtheorem{cor}[theorem]{Corollary}

 \newtheorem{problem}[theorem]{Problem}

\begin{document}

 \title{Lifting of nilpotent contractions}
\author{Tatiana Shulman}

\address{Mathematics Department, Moscow state aviation technological university, Moscow,
121552, Russia}

\email{tatiana$_-$shulman@yahoo.com}



\date{\today}

\maketitle



\section*{Introduction}
For the standard epimorphism from a $C^*$-algebra $A$ to  its
quotient $A/I$ by a closed ideal $I$, one may ask whether an element
$b$ in $A/I$ with some specific properties is the image of some
element $a$ in $A$ with the same properties. This is known as a
lifting problem that can be considered as a  non-commutative
analogue of extension problems for functions.

In the present work we deal with the following lifting problem. Let
$A$ be any $C^*$-algebra, $A/I$ --- any its quotient,  $b\in A/I$
--- a nilpotent contraction, that is $b^n=0$, for some fixed number
$n$, and $\|b\|\le 1$. Find an element $a\in A$ that is a
counter-image of $b$  and is also a nilpotent contraction: $a^n=0$
and $\|a\|\le 1$.

The problem can be reformulated in terms of representations of
relations. For the correct definition of a relation (or a system of
relations) see \cite{Loring}. Vaguely speaking a relation is an
equality of the form $f(x_1,...,x_n,x_1^*,...,x_n^*) = 0$ where $f$
is a non-commutative polynomial (an element of the corresponding
free algebra) and also inequalities of the form $\|x_i\|\le c$ or
$\|x_i\|<c$ are admitted. By a representation of a relation in a
$C^*$-algebra $A$ one means an $n$-tuple $(a_1,...,a_n)$ of elements
of $A$ that satisfy with their adjoints the given equality and
inequalities (if any).

A relation is liftable if for each its representation
$(b_1,...,b_n)$ in a quotient algebra $A/I$, there is a
representation $(a_1,...,a_n)$ in $A$ such that $b_i = a_i + I$.

So the question we consider is

\begin{problem}\label{problem} Are the relations
\begin{equation}\label{rel}
x^n=0 {\text\; and \;} \|x\| \le 1
\end{equation}
liftable?
\end{problem}

It was proved in \cite{Pedersen} that each nilpotent element can be
lifted to a nilpotent element. Later in \cite{Loring} it was shown
that among nilpotent preimages of a nilpotent element $b$ there are
elements of norm $\|b\|+\epsilon$, for any $\epsilon
>0$. This implies that the relations
\begin{equation}\label{rel1}
x^n=0 {\text\; and \;} \|x\| < 1
\end{equation}
 are
liftable.

Problem 1 was stated in \cite{Loring} in connection with the study
of projective $C^*$-algebras. A $C^*$-algebra $B$ is projective if
for any homomorphism $f$ from $B$ to a quotient $A/I$, there is a
homomorphism $g: B\to A$ with $f(b) = g(b) +I$, $b\in B$. Projective
$C^*$-algebras were introduced  in \cite{E-K}. There are very few
$C^*$-algebras that are known to be projective (\cite{Loring},
\cite{L-P}).

 A $C^*$-algebra $\mathcal{A}$ is
called a universal $C^*$-algebra of a  relation if there is a
representation $\pi_0$ of this relation in $\mathcal{A}$ such that
any representation of the relation factorizes uniquely through
$\pi_0$.   Note that not every system of relations has the universal
$C^*$-algebra, for example the relations (\ref{rel1}) don't have it
(the reason is that an infinite direct sum of strict contractions
needn't be a strict contraction).

It is not difficult to see that the universal $C^*$-algebra of a
relation is projective if and only if the relation is liftable. So
the question can be reformulated as follows: is the universal
$C^*$-algebra of the relations ({\ref{rel}) projective?

The aim of the present work is to give the positive answer to
Problem 1 and as a consequence to extend the list of known
projective $C^*$-algebras.

\section{M-ideals}

 We  recall some geometrical definitions (see, e.g., \cite{Werner}).

A closed subspace $Y$ of a Banach space $X$ is called an L-summand
(M-summand), if $X = Y \oplus Z$ for some subspace $Z\subset X$, and
$\|y+z\| = \|y\| + \|z\|$ (respectively $\|y+z\| =
\max\{\|y\|,\|z\|\})$ for all $y\in Y$, $z\in Z$. Furthermore $Y$ is
an M-ideal in $X$ if its annihilator $Y^{\bot}$ is an $L$-summand of
$X^*$.

Let $T$ be an operator on a Banach space $X$. Let us say that an
M-ideal $Y\subset X$ {\it dually reduces} $T$ if both summands of
the L-decomposition $X^* = Y^{\bot} + Z$ are invariant for $T^*$.
The definition is correct because a subspace can have only one
L-complement. Indeed if $X = Y + Z$ and $X = Y + W$ are
L-decompositions then for $z\in Z$ one has $z = y +w$ and $w = z+
(-y)$ whence $\|z\| = \|z\| + 2\|y\|$, $y=0$, $z\in W$.

\begin{theorem}\label{main} If an M-ideal $Y$ of $X$ dually reduces an operator $T$
then $\overline{TY}$ is an M-ideal of $\overline{TX}$.
 \end{theorem}
\begin{proof}
For any subspace $E\subset X$, one has the canonical  isometrical
isomorphism $i: E^* \to X^*/E^{\bot}$.  Namely, for any $h\in E^*$,
$i(h) = \tilde h + E^{\bot}$, where $\tilde h $ is any extension of
$h$ onto $X$. Hence
$$i(\overline{TX}^*)= X^*/(TX)^{\bot} = X^*/\ker(T^*).$$
Let $q$ denote the quotient map from $X^*$ to $X^*/\ker(T^*)$.

We have the L-sum decomposition $X^* = Y^{\bot} \oplus Z$; by the
assumption, both summands are $T^*$-invariant.

Let us prove that the decomposition $q(X^*) = q(Y^{\bot}) + q(Z)$
is an L-sum, that is for any $\xi = \eta + \zeta$, where $\eta\in
q(Y^{\bot})$, $\zeta\in q(Z)$, the equality
$\|\xi\|=\|\eta\|+\|\zeta\|$ holds.

 For any $\varepsilon>0$, there
is $f\in X^*$ such that $\xi = q(f)$ and $\|\xi\|\ge
\|f\|-\varepsilon$. There are  $f_1\in Y^{\bot}$, $f_2\in Z$ such
that $f=f_1+f_2$. Prove that $q(f_1) = \eta$, $q(f_2)=\zeta$.
Consider $g_1\in Y^{\bot}$, $g_2\in Z$ such that $\eta = q(g_1)$,
$\zeta = q(g_2)$. Then $$q(g_1+g_2) = \eta+\zeta = \xi=q(f) =
q(f_1+f_2),$$ whence $g_1-f_1+g_2-f_2\in ker( T^*)$. Hence
$T^*(g_1-f_1) = T^*(-g_2+f_2)$. Since $Y^{\bot}$ and $Z$ are
invariant subspaces for $T^*$, we get $T^*(g_1-f_1)\in Y^{\bot}$,
$T^*(-g_2+f_2)\in Z$, whence $T^*(g_1-f_1) = T^*(-g_2+f_2)=0$.
Hence $\eta= q(g_1)=q(f_1)$, $\zeta= q(g_2)=q(f_2)$ and we have
$$\|\xi\|\ge\|f\|-\varepsilon = \|f_1\|+\|f_2\|-\epsilon \ge
\|q(f_1)\|+\|q(f_2)\|-\varepsilon = \|\eta\|+\|\zeta\|-\epsilon.$$
Since it is true for every $\varepsilon >0$ and  since $\|\xi\|\le
\|\eta\|+\|\zeta\|$, the equality $\|\xi\|=\|\eta\|+\|\zeta\|$
holds. In particular it easily follows that $q(Y^{\bot})\bigcap q(Z)
= 0$ and  both subspaces $q(Y^{\bot})$ and $q(Z)$ are closed.

Thus $q(Y^{\bot})$ is an L-summand in $q(X^*)$. It remains to prove
that $q(Y^{\bot}) = i((\overline{TY})^{\bot})$, where by
$\overline{TY}^{\bot}$ we denote the annihilator of $TY$ in
$\overline{TX}^*$. Let $f\in Y^{\bot}$ and $f_1$ be its restriction
to $\overline{TX}$. Then, by definition, $q(f) = i(f_1)$.  For each
$y\in Y$, one has $f_1(Ty) = f(Ty)= (T^*f)(y) = 0$ because $T^*$
preserves $Y^{\bot}$. Hence $f_1$ belongs to $\overline{TY}^{\bot}$.
We proved that $q(Y^{\bot})\subset i(\overline{TY}^{\bot})$.

 Conversely, let
$h\in (\overline{TY})^{\bot}$.  There are $f_1\in Y^{\bot}$, $f_2\in
Z$ such that $i(h)=q(f_1)+q(f_2)$. Let $\tilde h $ be any extension
of $h$ onto $X$. Then $\tilde h -f_1-f_2\in ker(T^*)$ and $T^*\tilde
h - T^*f_1-T^*f_2=0$. Since, for any $y\in Y$, $(T^*\tilde
h)(y)=\tilde h(Ty)=h(Ty)=0$, $T^*\tilde h \in Y^{\bot}$. Since
$T^*\tilde h -T^*f_1\in Y^{\bot}$  and $T^*f_2\in Z$, we get
$T^*f_2=0$, that is $q(f_2)=0$. Hence $i(h)\in q(Y^{\bot})$ and
$q(Y^{\bot}) = i((\overline{TY})^{\bot})$.

We get that $i((\overline{TY})^{\bot})$ is an L-summand in
$i((\overline{TX}^*)$, or, equivalently, that  $\overline{TY}$ is an
M-ideal of $\overline{TX}$.
\end{proof}

For an arbitrary algebra $A$, an operator from $A$ to $A$ is called
an elementary operator if it is of the form $x\mapsto \sum_{i=1}^N
a_ixb_i$, where $a_i, b_i\in A$.

\begin{theorem}\label{oldmain} Let $A$ be a $C^*$-algebra, $I$  --- its ideal, $T:A\to A$ --- an elementary operator.
Then $\overline{TI}$ is an M-ideal in $ \overline{TA}$.
 \end{theorem}
 \begin{proof}
Recall first of all that the bidual $A^{**}$ of $A$ is a
$W^*$-algebra. Moreover the natural $A$-bimodule structure in $A^*$
($(af)(x) = f(xa)$, $(fa)(x) = f(ax)$) extends by the $\ast$-weak
continuity to $A^{**}$, so $A^*$ becomes an $A^{**}$-bimodule.

We may assume that the ideal $I$ is closed. It is well known
(\cite{Prosser}) that $I$ is an $M$-ideal in $A$. Moreover there
exists a central projection
 $p\in A^{**}$ such that $I^{\bot} = pA^*$ and $A^* = pA^* + (1-p)A^*$ is an L-sum.
 Since $T$ is elementary,  $T^*$ commutes with the multiplication by $p$, so both summands
 are invariant for $T^*$. This means that $I$ dually reduces $T$. It
 remains to apply Theorem \ref{main}.
 \end{proof}

\medskip

A subspace $Y$ of a Banach space $X$ is called proximinal (e.g.
\cite{Werner}) if for any $x\in X$ there is $y\in Y$ such that
$\|x-y\| = d(x, Y)$, or, equivalently, if for any $z\in X/Y$ there
is a lift $x$ of $z$ such that $\|z\|=\|x\|$.

\begin{cor}\label{prox}
Let $A$ be a $C^*$-algebra, $T$ an elementary operator on $A$. Then,
for each ideal $I$ of $A$, the subspace $\overline{TI}$ is
proximinal in $\overline{TA}$.
\end{cor}
\begin{proof} By (\cite{Werner}, Theorem II.1.1),
all M-ideals are proximinal subspaces. So it suffices to use
Theorem \ref{oldmain}.
\end{proof}

\section{Projectivity}

Below we use the following notation.

By $M(I)$ we denote the multiplier $C^*$-algebra of a $C^*$-algebra
$I$.

For elements $x,y$ of a $C^*$-algebra $A$,  $x<<y$ means $xy = yx =
x$.

For a closed ideal $I$ of $A$ we denote by $a\to \dot a$ the
standard epimorphism from $A$ to $A/I$. We say that $a$ is a lift of
$b$ if $\dot{a} = b$.

\begin{theorem} The universal
 $C^*$-algebra of the relations (\ref{rel}) is
projective.
\end{theorem}
\begin{proof} By
Theorem 10.1.9 of \cite{Loring}, it suffices to prove that  any
$\ast$-homomorphism from the  universal $C^*$-algebra of the
relations (\ref{rel}) to any quotient $M(I)/I$  lifts to a
$\ast$-homomorphism to $M(I)$, or, equivalently, any nilpotent
contraction in $M(I)/I$ lifts to a nilpotent contraction in $M(I)$.

Fix a nilpotent contraction $b$ in $M(I)/I$: $b^n = 0$. If $\|b\|<1$
then it can be lifted to a nilpotent contraction by  Theorem 12.1.6
of \cite{Loring}. Thus we assume $\|b\|=1$. It is proved in
\cite{Pedersen} that there are elements
$$0\le p_{n-1} <<q_{n-1} << p_{n-2} << \ldots <<q_2<<p_1<<q_1 \le
q_0=1$$ in $M(I)$ such that $\sum_{j=1}^{n-1} (q_{j-1}-q_j)ap_j$ is
a nilpotent (of order $n$) lift of $b$, for any lift $a$ of $b$.
  Let
$$E = \{\sum_{j=1}^{n-1} (q_{j-1}-q_j)ap_j\;|\;\;\dot a = b\}.$$
Note that $\overline E$ also consists of nilpotent lifts of $b$.
Indeed if $a_k\to a$, $a_k^n=0$, $\dot a_k = b$ then $a^n=0$ and
$\dot a = b$. We are going to find in $\overline E$ an element of
norm 1. Define an operator $T:M(I)\to M(I)$ by $$Tx =
\sum_{j=1}^{n-1} (q_{j-1}-q_j)xp_j.$$ Let $a_0$ be some fixed lift
of $b$. Then
$$E = \{Ta_0 +Ti\;|\;\;i\in I\},\;\; \overline E = \{Ta_0
+x\;|\;\;x\in \overline{TI}\}.$$ Since $T$ is an elementary
operator, $\overline{TI}$ is proximinal in $\overline{TM(I)}$  by
Theorem \ref{prox} .

 This implies that in $\overline E$ there is an element of norm 1.
 Indeed, since  $\overline E$ is the set of all lifts of the element
$Ta_0 + \overline{TI}\in \overline{TM(I)}/\overline{TI}$, it is
enough to prove that the norm $\|Ta_0 + \overline{TI}\|$ of the
element $Ta_0 + \overline{TI}$ in $\overline{TM(I)}/\overline{TI}$
is equal to $1$.  We have
$$\|Ta_0 + \overline{TI}\| = \inf_{x\in \overline{TI}} \|Ta_0+x\|\le
\inf_{i\in I} \|Ta_0+Ti\| = \inf_{x\in E} \|x\| = 1$$  by
 Theorem 12.1.6 of \cite{Loring}. It is obvious that $\|Ta_0 +
\overline{TI}\|\ge 1$ because
  $(Ta_0+x)$ is a lift of $b$ for any $x\in \overline{TI}$ whence
$$\|Ta_0 + \overline{TI}\| = \inf_{x\in \overline{TI}} \|Ta_0+x\|
\ge \|b\|=1.$$ Thus in $\overline E$ there is an element of norm 1,
that is a nilpotent lift of $b$ of norm $1$.
\end{proof}

\end{document}